\documentclass{amsart}
\usepackage{amsmath}
\usepackage{color}
\usepackage{mathtools}

\newtheorem{theorem}{Theorem}[section]
\newtheorem{lemma}[theorem]{Lemma}
\newtheorem{proposition}[theorem]{Proposition}

 \theoremstyle{definition}
\newtheorem{definition}[theorem]{Definition}

\theoremstyle{remark}

\numberwithin{equation}{section}

\begin{document}

\title[A topological rigidity theorem on noncompact Hessian manifolds]
{A topological rigidity theorem on noncompact Hessian manifolds}

\author{Hanzhang Yin}
\address{School of Mathematics, Harbin Institute of Technology,
         Harbin, Heilongjiang 150001, China}
\email{YinHZ@hit.edu.cn}

\author{Bin Zhou}
\address{School of Astronautics, Harbin Institute of Technology,
         Harbin, Heilongjiang 150001, China}
\email{binzhou@hit.edu.cn}
\thanks{The second author is supported by the National Natural Science Foundation of China for Distinguished Young Scholars under (Grant No. 62125303), the National Natural Science Foundation of China "Qisun Ye" Science Foundation (Grant No. U2441243).}



\begin{abstract}

In this work, we obtain a short time solution for a geometric flow on noncompact affine Riemannian manifolds. Using this
result, we can construct a Hessian metric with nonnegative bounded Hessian sectional curvature on some Hessian manifolds with nonnegative Hessian sectional curvature.
Our results can be regarded as a real version of Lee-Tam \cite{LT20}.
As an application, we prove that a complete noncompact Hessian manifold with nonnegative Hessian sectional curvature
is diffeomorphic to $\mathbb{R}^n$ if its tangent bundle has maximal volume growth. This is an improvement of Theorem 1.3 in Jiao-Yin \cite{JY25}.


\end{abstract}

\maketitle

\section{Introduction}

Mirghafouri-Malek \cite{r12} studied a geometric flow on compact Hessian manifolds. In this paper, we want to study the geometric flow and the topological rigidity theorem on noncompact Hessian manifolds. We will construct a solution to the flow starting from the given complete
metric $g_0$ with possibly unbounded curvature at infinity, to deal with the possible bad behavior of the metric $g_0$ at infinity we need to make a conformal change to it outside some large compact set $S$. Therefore, we need to study the geometry flow on general affine manifolds.

Let $(M, \nabla, g_0)$ be a complete noncompact affine Riemannian manifold, with $\nabla$ a flat, torsion free connection
and $g_0$ a Riemannian metric on $M$. The following flow, was first studied by Jiao-Yin \cite{JY25}: starting at ${g}_{0}$ is given by
\begin{equation}
\label{e1.2}
\left\{ \begin{aligned}
   &\frac{\partial }{\partial t}{{g}}=-\beta (g)  \\
      &g(0)={{g}_{0}}  \\
\end{aligned} \right.
\end{equation}
on $M$, where $\beta (g)$ is a form associated to the Riemannian metric $g$, which in local affine coordinates is given by
\[
{{\beta }_{ij}}(g)=-{{\partial }_{i}}{{\partial }_{j}}\log \det [g(t)]
\]
In the case when $g_0$ is Hessian, i.e., $g_0$ can be locally expressed
by $g_0=\nabla d\varphi$,
the evolved metrics $g(t)$ along the flow remain Hessian (seeing \cite{r12}) and $-\frac{\beta}{2}$ is called the second Koszul form
of $g(t)$ (see Definition \ref{Kos}).

The affine manifold $M$ is a smooth manifold that admits local coordinate systems such that the changes of the local coordinate systems are affine transformations. There are several important geometric problems on affine manifolds, such as Chern's conjecture (see \cite{Klingler17}).
The Hessian manifold is an affine manifold compatible with a Hessian metric, which is also called affine K\"{a}hler manifolds. It plays a key role in mathematical physics, statistics and information geometry. The reader is referred to the introduction of \cite{JY25} and \cite{r11} for more details on
the geometry and applications of affine/Hessian manifolds.

Affine/Hessian manifolds are closely connected to Hermitian/K\"{a}hler manifolds. Indeed, the
tangent bundle of an affine Riemannian manifold has a complex structure, and we can construct a Hermitian metric. In particular, the Hermitian metric
is K\"{a}hler if and only if the Riemannian metric $g_0$ is Hessian (see Section 2 for more details). Therefore, when $g_0$ is Hessian/Riemannian, \eqref{e1.2} may be regarded as a real analogue of K\"{a}hler/Chern-Ricci flow. In this paper, we apply certain methods from the Chern-Ricci flow to study equation \eqref{e1.2}. The reader is referred to a survey of Tosatti-Weinkove \cite{TW21} and the introduction of \cite{JY25} for more details on the Chern-Ricci flow.


In this paper, we will discuss the existence of \eqref{e1.2} on noncompact affine Riemannian manifolds.
Let $S_A$ be the supremum of $S>0$ so that the flow \eqref{e1.2} has a solution $g(t)$ with initial data $g_0$ such that $g(t)$ is uniformly equivalent to $g_0$ in $M\times [0,S]$
and $S_B$ the supremum of $S>0$ such that there is a smooth bounded function $u$ satisfying
\begin{equation}
\label{e5.1}
g_0-S\beta(g_0)+\nabla du\geq \theta g_0
\end{equation}
for some $\theta>0$.

In \cite{JY25}, Jiao-Yin obtained the following criteria of existence time.
\begin{theorem}
\label{t1.1}
Let $(M, \nabla, g_0)$ be a complete noncompact affine Riemannian manifold. Assume that

\textbf{\emph{(i)}}
\begin{minipage}[t]{0.9\linewidth}
There exists a smooth real function $\rho$ which is uniformly equivalent to the distance function from a fixed point such that $|d\rho|_{g_0}$, $|\nabla d\rho|_{g_0}$ are uniformly bounded.
\end{minipage}

\textbf{\emph{(ii)}}
\begin{minipage}[t]{0.9\linewidth}
For any point $p \in M^n$, there exists a local coordinates $\{x^1,\ldots ,x^n\}$ around $p$ such that $(g_0)_{ij}(p)=\delta_{ij}$, $\partial_k(g_0)_{ij}(p)$ and $\partial_i(\partial_j (g_0)_{kl}-\partial_l (g_0)_{kj})(p)$ are uniformly bounded, $\partial_l \partial_k(g_0)_{ij}(p)$ are uniformly bounded from above and all the changes
of these local coordinate systems are affine.
\end{minipage}

Then $S_A=S_B$.
\end{theorem}

In this work, we want to weaken the conditions in the above theorem, such as removing the restriction on the first derivative of the metric. We obtain the following:
\begin{theorem}\label{t1.2}
Let $(M, \nabla, g_0)$ be a complete noncompact affine Riemannian manifold. Assume that

\hspace*{0.1cm}\textbf{\emph{(i)}}
\begin{minipage}[t]{0.9\linewidth}
On $(TM,J_\nabla,g_0^T)$, there exists a smooth real function $\rho$ which is uniformly equivalent to the distance function from a fixed point such that $|\partial\rho|_{g_0^T}$, $|\partial\bar{\partial}\rho|_{g_0^T}$ are uniformly bounded.
\end{minipage}\\

\textbf{\emph{(ii)}}
\begin{minipage}[t]{0.9\linewidth}
$|T_0|_{g_0^T}$ and $|\bar{\partial}T_0|_{g_0^T}$ are uniformly bounded. The bisectional curvature of $g_0^T$ is uniformly bounded from below.
\end{minipage}\\

Then $S_A=S_B$.
\end{theorem}
When $g_0$ is a Hessian metric, by Theorem \ref{t1.2}, we have
\begin{theorem}\label{thm2}
Let $(M,g_0)$ be a complete noncompact Hessian manifold with bounded Hessian curvature. Then $S_A=S_B$.
\end{theorem}

The classical uniformization theorem of surfaces states that there exists a metric with constant Gaussian curvature on every surface. Equivalently, in the language of complex analysis, every simply connected Riemann surface is conformally equivalent to one of three Riemann surfaces: the open unit disk, the complex plane, or the Riemann sphere. In particular, a complete noncompact Riemannian surface with positive Gaussian curvature is conformally equivalent to the complex plane. In
 higher dimensions, a longstanding conjecture of Yau \cite{Yau89} states that
a complete noncompact K\"{a}hler manifold with positive holomorphic
bisectional curvature is biholomorphic to $\mathbb{C}^n$. This conjecture remains unresolved to this day, and even proving that manifolds satisfying the conditions of the conjecture are simply connected is still an open problem.

Chau-Tam \cite{Tam06} studied the asymptotic behavior of the K\"{a}hler-Ricci flow. They proved that a complete noncompact K\"{a}hler manifold with nonnegative bounded holomorphic bisectional curvature and maximal volume growth is biholomorphic to complex Euclidean space $\mathbb{C}^n$. It was proved by Liu \cite{Liu19} that a complete noncompact
K\"{a}hler manifold with nonnegative bisectional curvature and maximal volume growth is biholomorphic to
$\mathbb{C}^n$. The same result was proved by Lee-Tam \cite{LT20} later using the Chern-Ricci flow. They constructed a short time solution to the K\"{a}hler-Ricci flow on some K\"{a}hler manifolds, which does not need to have bounded holomorphic bisectional curvature, and proved that a complete noncompact K\"{a}hler manifold with nonnegative holomorphic bisectional curvature and maximal volume growth also has a metric with nonnegative bounded holomorphic bisectional curvature and maximal volume growth.

Similar to \cite{LT20}, we obtain the following short time solution to \eqref{e1.2}:
\begin{theorem}\label{t4.1}
Suppose $(M^n,g_0)$ is a complete noncompact Hessian manifold with dimension $n$ with nonnegative Hessian sectional curvature such that $V_0(x,1)\geq v_0>0$ for some $v_0>0$ for all $x\in TM$. There exist $S=S(n,v_0)>0$, $a(n,v_0)>0$ depending only on $n$, $v_0$ such that the geometric flow \ref{e1.2} has a complete solution $g(t)$ on $M\times [0,S]$ which satisfies the following:

\hspace*{0.2cm}\emph{(i)}
\[|Q_{ijkl}|_{g(t)}\leq \frac{a}{t}\]
\hspace*{1.1cm}on $M\times (0,S]$.

\hspace*{0.1cm}\emph{(ii)} $g(t)$ has nonnegative Hessian sectional curvature.

\emph{(iii)} If $g^T(0)$ has maximal volume growth, then $g^T(t)$ also has maximal volume\\ \hspace*{1.1cm}growth.
\end{theorem}
Finally, we prove the main theorem of this paper, which can be regarded as a
real analogue of the uniformization theorem on certain complex manifolds (see \cite{LT20}).
\begin{theorem}\label{t1.5}
Let $(M^n,g_0)$ be a complete noncompact Hessian manifold with nonnegative Hessian sectional curvature, and its tangent bundle $(TM,g_0^T)$ has maximal volume growth. Then $M$ is diffeomorphic to $\mathbb{R}^n$.
\end{theorem}
The above theorem weakens the conditions in Theorem 1.3 of \cite{JY25}.


The rest of this paper is organized as follows: in Section 2, we provide some preliminaries which may be used later.
In Section 3, we establish the \emph{a priori} estimates up to $C^2$
for the flow \eqref{e1.2} and obtain the existence criteria (Theorem \ref{t1.2} and Theorem \ref{thm2}) for \eqref{e1.2}.
In Section 4, we obtain the short time solution to \eqref{e1.2} (Theorem \ref{t4.1}) and prove the topological rigidity theorem for noncompact Hessian manifolds (Theorem \ref{t1.5}).

\textbf{Acknowledgement.} The first author wishes to thank Professor Jiao Heming for suggesting him to read \cite{r13,r3}
and to study some elliptic/parabolic equations on Hessian manifolds.

\section{Preliminaries}

In this section, we provide some basic results and preliminaries which may be used in the following sections. Most of these can be found in \cite{JY25}, for the convenience of the reader, we rewrite them here.

First, as discussed in \cite{JY25}, \eqref{e1.2} can be rewritten as a scalar equation, which is a kind of parabolic Monge-Amp\`{e}re equations:
\begin{equation}
\label{e1.3}
\left\{ \begin{aligned}
   &\frac{\partial \varphi }{\partial t}=\log \frac{\det ({{g}_{0}}-t\beta ({{g}_{0}})+\nabla d\varphi )}{\det ({{g}_{0}})},  \\
   &{{g}_{0}}-t\beta ({{g}_{0}})+\nabla d\varphi >0,  \\
   &\varphi (0)=0.  \\
\end{aligned} \right.
\end{equation}
where $\varphi$ is an unknown function. The relationship between $\varphi$ and $g$ is described as follows:
\[g(t):={g}_{0}-t\beta ({{g}_{0}})+\nabla d\varphi,\]
\[\varphi (x,t)=\int_{0}^{t}{\log \frac{\det g(x,s)}{\det {{g}_{0}}(x)}}ds.\]

Now we introduce fundamental definitions and key results on affine and Hessian manifolds. Most of them can be found in \cite{r11}.


\begin{definition} An affine manifold $(M,\nabla)$ is a differentiable manifold endowed with a flat connection $\nabla$.
\end{definition}
The subsequent statement is an immediate consequence of the affine manifold definition.
\begin{proposition}
\label{prop1}
\

\emph{(1)}
\begin{minipage}[t]{0.92\linewidth}
Suppose that $M$ admits a flat connection $\nabla$. Then there exist local coordinate systems on $M$ such that $\nabla_{\partial /\partial x^i} \partial /\partial x^j=0$. The changes between such coordinate systems are affine transformations.
\end{minipage}

\emph{(2)}
\begin{minipage}[t]{0.92\linewidth}
Conversely, if $M$ admits local coordinate systems such that the changes of the local coordinate systems are affine transformations, then there exists a flat connection $\nabla$ satisfying $\nabla_{\partial /\partial x^i} \partial /\partial x^j=0$ for all such local coordinate systems.
\end{minipage}
\end{proposition}
Let $(M, \nabla, g)$ be an affine manifold with a Riemannian metric $g$, where $\hat{\nabla}$ denotes the Levi-Civita connection of
$(M,g)$. Defining $\gamma=\hat{\nabla}-\nabla$. Since both connections $\nabla$ and $\hat{\nabla}$ are torsion-free, we obtain
\[\gamma_X Y=\gamma_Y X\]
Furthermore, in affine coordinate systems, the components $\gamma_{\;jk}^i$ of $\gamma$
coincide with the Christoffel symbols $\Gamma_{\;jk}^i$ associated with the Levi-Civita connection $\hat{\nabla}$.

For the pair $(g,\nabla)$, the Hessian curvature tensor is given by $Q=\nabla \gamma$.
A Riemannian metric $g$ on a affine manifold $(M,\nabla)$ is called Hessian if and only if it admits local expression through $g=\nabla d\varphi$. An affine manifold endowed with a Hessian metric is called a Hessian manifold.

Given a flat connection $\nabla$, a local coordinate system $\{x^1,\ldots,x^n\}$ satisfying $\nabla_{\frac{\partial}{\partial x^i}} \frac{\partial}{\partial x^j}=0$ is called an affine coordinate system with respect to $\nabla$.
Let $(M,\nabla,g)$ be a Hessian manifold and its metric $g$ admits local representation
\[g_{ij}=\frac{\partial^2 \phi}{\partial x^i \partial x^j}\]
where $\{x^1,\ldots,x^n\}$ is an affine coordinate system with respect to $\nabla$. The next two results are established in \cite{r11}.

\begin{proposition}\label{p2.3}
Let $(M,\nabla)$ be an affine manifold and $g$ a Riemannian metric on $M$. Then the following are equivalent:

\emph{(1)} $g$ is a Hessian metric

\emph{(2)} $(\nabla_X g)(Y,Z)=(\nabla_Y g)(X,Z)$

\emph{(3)} $\dfrac{\partial g_{ij}}{\partial x^k}=\dfrac{\partial g_{kj}}{\partial x^i}$

\emph{(4)} $g(\gamma_X Y,Z)=g(Y,\gamma_X Z)$

\emph{(5)} $\gamma_{ijk}=\gamma_{jik}$

\end{proposition}

\begin{proposition}\label{p2.4}
 Let $\hat{R}$ be the Riemannian curvature of a Hessian metric $g=\nabla d\phi$ and $Q=\nabla \gamma$ be the Hessian curvature tensor for $(g,\nabla)$. Then

\emph{(1)} $Q_{ijkl}=\dfrac{1}{2}\dfrac{\partial^4 \phi}{\partial x^i \partial x^j \partial x^k \partial x^l}-\dfrac{1}{2}g^{pq}\dfrac{\partial^3 \phi}{\partial x^i \partial x^k \partial x^p}\dfrac{\partial^3 \phi}{\partial x^j \partial x^l \partial x^q}$

\emph{(2)} $\hat{R}(X,Y)=-[\gamma_X,\gamma_Y],\;\;\;\hat{R}^i_{\;jkl}=\gamma^i_{\;lm}\gamma^m_{\;jk}-\gamma^i_{\;km}\gamma^m_{\;jl}.$

\emph{(3)}
$\hat{R}_{ijkl}=\dfrac{1}{2}(Q_{ijkl}-Q_{jikl})$

\end{proposition}

\begin{definition}
\label{Kos}
Let $(M,\nabla,g)$ be a Hessian manifold and
$v$ the volume element of $g$. The first Koszul form $\alpha$ and the second Koszul form $\kappa$ for $(\nabla,g)$ are defined by
\[\nabla_X v=\alpha(X)v\]
\[\kappa=\nabla \alpha\]
\end{definition}
It follows that
\[\alpha(X)={\rm Tr} {\gamma_X}\]
and
\[\alpha_i=\frac{1}{2}\frac{\partial \log \det[g_{pq}]}{\partial x^i}=\gamma_{\;ki}^k\]
\[\kappa_{ij}=\frac{\partial \alpha_i}{\partial x^j}=\frac{1}{2}\frac{\partial \log \det[g_{pq}]}{\partial x^i \partial x^j}\]
locally.


Now we introduce the relationship between affine structures and complex structures, this is the core technology of the paper.
Let $(M,\nabla)$ be a flat manifold and $TM$ be the tangent bundle of $M$ with projection $\pi:TM\rightarrow M$. For an affine coordinate system $\{x^1,\ldots,x^n\}$ on $M$, we define
\begin{equation}
z^j=\xi^j+\sqrt{-1}\xi^{n+j}
\end{equation}
where $\xi^i=x^i\circ \pi$ and $\xi^{n+j}=dx^i$. Then $n$-tuples of functions given by $\{z^1,\ldots,z^n\}$ yield holomorphic coordinate systems on $TM$. We denote by $J_\nabla$ the complex structure tensor of the complex manifold $TM$. For a Riemannian metric $g$ on $M$ we put
\begin{equation}\label{e2.2}
g^T=\sum_{i,j=1}^{n}(g_{ij}\circ \pi)dz^id\overline{z}^j.
\end{equation}
Then $g^T$ is a Hermitian metric on the complex manifold $(TM,J_\nabla)$.
\begin{proposition} (\cite{r11}) \label{p2.6}
Let $(M,\nabla)$ be a flat manifold and $g$ a Riemannian metric on $M$. Then the following conditions are equivalent.

\noindent\emph{(1)} $g$ is a Hessian metric on $(M,\nabla)$.

\noindent\emph{(2)} $g^T$ is a K{\"a}hler metric on $(TM,J_\nabla)$.

\end{proposition}
Furthermore, we have
\begin{proposition} (\cite{r11}) \label{p2.7}
Let $(M, \nabla, g)$ be a Hessian manifold and $R^T$ be the Riemannian curvature tensor of the K{\"a}hler manifold $(TM,J,g^T)$. Then we have
\begin{equation}
\label{sec}
R_{i\bar{j}k\bar{l}}^T=-\frac{1}{2}Q_{ijkl}\circ \pi.
\end{equation}
Let $R_{i\bar{j}}^T$ be the Ricci tensor of the K{\"a}hler manifold $(TM,J,g^T)$. Then we have
\[R_{i\bar{j}}^T=\frac{1}{4}\beta_{ij}\circ \pi.\]
\end{proposition}
\begin{definition}
\label{bisectional}
Let $(M,g_0)$ be a Hessian manifold. Let $Q$ be the Hessian curvature tensor for $(g_0,\nabla)$. By \eqref{sec}, in order to compatible
to the study of complex geometric flows, $g_0$ is said to have Hessian sectional curvature bounded bolow by $-K$, if at any point and for any orthogonal frame, $Q_{iijj}\leq K$. The Hessian sectional curvature of $g_0$ is bounded above by $K$ is defined similarly. In particular, $g_0$ is said to have nonnegative nonnegative Hessian sectional curvature if at any point and for any orthogonal frame, $Q_{iijj}\leq 0$.
\end{definition}

\section{A geometric flow on noncompact affine Riemannian manifolds}

In \cite{JY25}, we obtain the a priori estimates for \eqref{e1.3} by studying the equation itself, because we cannot choose a geodesic coordinate system in affine manifolds, there exists a condition on the first-order derivatives of the metric $(g_0)_{ij}$(see \cite{JY25} Theorem 1.1). To remove the above condition, in this section, we shall explore the relationship between the flow \eqref{e1.2} and the Chern-Ricci flow, and obtain the a priori estimates for \eqref{e1.3} on noncompact affine Riemannian manifolds using the a priori estimates for the Chern-Ricci flow, as shown in \cite{LT20}. At last, we derive the existence criteria for the geometric flow \eqref{e1.2}.

Let $(M,\nabla, g_0)$ be a complete affine Riemannian manifold of dimension $n$, and $g(t)$ is a solution of \eqref{e1.2}.
We find, by \eqref{e2.2},
\[
g^T (t) := \sum_{i,j=1}^{n}(g_{ij} (t)\circ \pi)dz^id\overline{z}^j
\]
and
\begin{equation}
\begin{aligned}
R_{i\overline{j}} (g(t)^T)&=-\partial_i \partial_{\bar{j}}\log\det(g(t)^T)\\
&=-\frac{1}{2}(\frac{\partial}{\partial \xi^i}-\sqrt{-1}\frac{\partial}{\partial \xi^{n+j}})\frac{1}{2}(\frac{\partial}{\partial \xi^j}-\sqrt{-1}\frac{\partial}{\partial \xi^{n+j}})\log\det(g(t))\circ \pi\\
&=-\frac{1}{4}\frac{\partial^2}{\partial \xi^i\partial\xi^j}\log\det(g(t))\circ \pi\\
&=\frac{1}{4}\beta_{ij}\circ \pi,
\end{aligned}
\end{equation}
so $\widetilde{g}(t)=g^T (t/4)$ is a solution to the Chern-Ricci flow on the tangent bundle $(TM,J_\nabla,g_0^T)$
\begin{equation}\label{e3.2}
\left\{ \begin{aligned}
   &\frac{\partial }{\partial t}{\widetilde{g}_{i\overline{j}}}(t)=-R_{i\overline{j}} (\widetilde{g}(t))  \\
      &\widetilde{g}(0)= g_0^T.
\end{aligned} \right.
\end{equation}
By \cite{LT20}, \eqref{e3.2} is equivalent to the following parabolic complex Monge-Amp\`{e}re equation:
\begin{equation}
\left\{ \begin{aligned}
   &\frac{\partial \psi }{\partial t}=\log \frac{\det (g_0^T-t{\rm Ric}(g_0^T)+\sqrt{-1}\partial\bar{\partial}\psi)}{\det (g_0^T)},  \\
   &\psi(0)=0.  \\
\end{aligned} \right.
\end{equation}
Where
\begin{equation}
\begin{aligned}
\psi (x,t)&=\int_{0}^{t}{\log \frac{\det \widetilde{g}(x,s)}{\det {{g}^T_{0}}(x)}}ds\\
&=\int_{0}^{t}{\log \frac{\det g(x,s/4)}{\det {{g}_{0}}(x)}}ds\circ\pi\\
&=4\varphi(t/4)\circ\pi,
\end{aligned}
\end{equation}

We introduce the following conditions on a complete noncompact affine Riemannian manifold $(M^n,g_0)$ as \cite{LT20}:

\textbf{(a1)}
\begin{minipage}[t]{0.9\linewidth}
On $(TM,J_\nabla,g_0^T)$, there exists a smooth real function $\rho$ which is uniformly equivalent to the distance function from a fixed point such that $|\partial\rho|_{g_0^T}$, $|\partial\bar{\partial}\rho|_{g_0^T}$ are uniformly bounded.
\end{minipage}\\

\textbf{(a2)}
\begin{minipage}[t]{0.9\linewidth}
There is a smooth bounded function $u$ and positive constants $S,\theta$ such that
\[g_0-S\beta(g_0)+\nabla du\geq \theta g_0\]
\end{minipage}\\

Assume that $g_0$ satisfies \textbf{(a1)} and \textbf{(a2)}.
Let $u$ and $S$ be defined in \textbf{(a2)}. Suppose that $g(t)$ is a solution to the geometric flow \eqref{e1.2} on $M\times[0,S_1]$
by choosing $S_1<S$ such that $g(t)$ is uniformly equivalent to $g_0$ on $M\times[0,S_1]$.
Let $\varphi$ be the corresponding solution to \eqref{e1.3}. By \cite{LT20} Lemma 3.4 and 3.5, we have Lemma \ref{l3.1} and \ref{l3.2}.
\begin{lemma}\label{l3.1} Suppose that $g_0$ satisfies \emph{\textbf{(a1)}}, \emph{\textbf{(a2)}} and the curvature of $g_0^T$ with ${\rm BK}(g_0^T)\geq -K$ for any unitary frame at any point in $TM$. Then there is a constant $c(n)>0$ depending only on $n$ such that for $t\leq S_1$,\\
\hspace*{0.5cm}\emph{(i)} $\varphi\leq (\log(1+c(n)KS_1)^n+1)t$\\
\hspace*{0.4cm}\emph{(ii)} $\dot{\varphi}\leq (\log(1+c(n)KS_1)^n+1)+n.$\\
\hspace*{0.3cm}\emph{(iii)}
\[\dot{\varphi}(x,t)\geq \frac{1}{S-S_1}\biggl[\underset{M}{\inf}\;u-\underset{M}{\sup}\;u-(\log(1+c(n)KS_1)^n+1+n)t\biggl],\]
\hspace*{1.1cm}and
\[\varphi(x,t)\geq \frac{1}{S-S_1}\int_0^t\biggl[\underset{M}{\inf}\;u-\underset{M}{\sup}\;u-(\log(1+c(n)KS_1)^n+1+n)s\biggl]ds.\]
where $u$ is the smooth bounded function defined in \emph{\textbf{(a2)}}.
\end{lemma}

\begin{lemma}\label{l3.2} Let $g_0$ be as in Lemma \ref{l3.1}. Suppose in addition $|T_0|_{g_0^T}^2$ and $|\bar{\partial}T_0|_{g_0^T}$ are uniformly bounded by $K_1$. Then there are constants $c_1(n)$, $c_2(n)$ depending only on $n$ such that for $t\leq S_1<S_2<S$,
the solution $g(t)$ of \eqref{e1.2} on $M\times[0,S_1]$ which is uniformly equivalent to $g_0$, satisfies
\[{\rm tr}_{g_0}g\leq \exp\biggl[\log\biggl(\frac{1}{2}(c_1+(c_1^2+c_2K_1^2A)^{\frac{1}{2}})+A\biggl]\]
on $M\times[0,S_1]$, where
\[A=\alpha^{-1}(2\mathfrak{m}+1)^2(c_1(K+K_1)+1) \]
and $\alpha=1-\frac{S_2}{S}$ and $\mathfrak{m}=\sup_{M\times[0,S_1]}|(S_2 -t)\dot{\varphi}+\varphi+nt-\frac{S_2}{S}u|$.
\end{lemma}

Then we consider the existence criteria for the geometric flow \eqref{e1.2}. Let $S_A$ be the supremum of $S>0$ so that the flow \eqref{e1.2} has a solution $g(t)$ with initial data $g_0$ such that $g(t)$ is uniformly equivalent to $g_0$ in $M\times [0,S]$
and $S_B$ the supremum of $S>0$ such that there is a smooth bounded function $u$ satisfying
\begin{equation}
\label{e5.1}
g_0-S\beta(g_0)+\nabla du\geq \theta g_0
\end{equation}
for some $\theta>0$.

\begin{proof}[Proof of Theorem \ref{t1.2}]
The proof of this theorem is similar to the proof of Theorem 1.1 in \cite{JY25}, we need only make a small change. Indeed, it is easy to find that $(M, g(t))$ is a submanifold of $(TM, g^T(t))$, and on $M$, $\rho|_{M}$ is a smooth real function which is uniformly equivalent to the distance function from a fixed point such that $|d\rho|_{g_0}$, $|\nabla d\rho|_{g_0}$ are uniformly bounded, this is \textbf{(a1)} condition in \cite{JY25}. Then, we can use Lemma \ref{l3.1} and \ref{l3.2} instead of Lemma 4.4 and 4.5 in \cite{JY25} to prove the theorem.
\end{proof}

\begin{theorem}\label{t3.4}
Let $(M, \nabla, g_0)$ be a complete noncompact affine Riemannian manifold. Suppose the curvature of $g_0^T$, $|T_0|_{g_0^T}^2$ and $|\bar{\partial}T_0|_{g_0^T}$ are uniformly bounded by $K>0$. Suppose also the Riemannian curvature of $g_0^T$ is bounded. Then there is a constant $\alpha(n)>0$ depending only on $n$ such that \eqref{e1.2} has a smooth solution $g(t)$ on $M\times[0,\alpha K^{-1}]$ such that $\alpha g_0\leq g(t)\leq \alpha^{-1}g_0$ on $M\times[0,\alpha K^{-1}]$.
\end{theorem}
\begin{proof}
Since the Riemannian curvature of $g_0^T$, and $|T_0|_{g_0^T}^2$ are uniformly bounded, we conclude that $g_0^T$ satisfies \textbf{(a1)} by \cite{LT20}. Since the curvature of $g_0^T$ is bounded from above by $K$, there exists a constant $c_1(n)$ depending only on $n$ such that
\[g_0-2c_1K^{-1}\beta(g_0)\geq \theta g_0.\]
Since the curvature of $g_0^T$ is also bounded from below by $-K$, and $|T_0|_{g_0^T}$, $|\bar{\partial}T_0|_{g_0^T}$ are uniformly bounded, by Theorem \ref{t1.2}, we have that there is a solution $g(t)$ of \eqref{e1.2} on $M\times[0,c_1 K^{-1}]$ such that $g(t)$ is uniformly equivalent to $g_0$. Then by Theorem 4.2 of \cite{LT20}, there is a smooth solution $\widetilde{g}(t)=g^T (t/4)$ of \eqref{e3.2} on $TM\times[0,\widetilde{\alpha} K^{-1}]$ such that $\widetilde{\alpha} g_0^T\leq \widetilde{g}(t)\leq \widetilde{\alpha}^{-1}g_0^T$ on $TM\times[0,\widetilde{\alpha} K^{-1}]$, where $0<\widetilde{\alpha}\leq 4c_1$ is a constant depending only on $n$. Then we can take $\alpha=\widetilde{\alpha}/4$ to prove this theorem.
\end{proof}

Applying Theorem \ref{t1.2} to Hessian manifolds, we can prove Theorem \ref{thm2}:

\begin{proof}[Proof of Theorem \ref{thm2}]By Proposition \ref{p2.4} and Proposition \ref{p2.7}, the Hessian manifolds in Theorem \ref{thm2} satisfy {\textbf{(i)}} and {\textbf{(ii)}} in Theorem \ref{t1.2}.
Thus, Theorem \ref{thm2} follows.
\end{proof}

\section{Existence of a geometric flow on Hessian manifolds}

In this section, we use the geometric flow \ref{e1.2} and the methods in \cite{LT20} to construct a solution of \ref{e1.2} starting from the given complete Hessian metric $g_0$ with possibly unbounded curvature at infinity. We will prove the following:
\begin{theorem}\label{t4.1}
Suppose $(M^n,g_0)$ is a complete noncompact Hessian manifold with dimension $n$ with nonnegative Hessian sectional curvature such that $V_0(x,1)\geq v_0>0$ for some $v_0>0$ for all $x\in TM$. There exist $S=S(n,v_0)>0$, $a(n,v_0)>0$ depending only on $n$, $v_0$ such that the geometric flow \ref{e1.2} has a complete solution $g(t)$ on $M\times [0,S]$ which satisfies the following:

\hspace*{0.2cm}\emph{(i)}
\[|Q_{ijkl}|_{g(t)}\leq \frac{a}{t}\]
\hspace*{1.1cm}on $M\times (0,S]$.

\hspace*{0.1cm}\emph{(ii)} $g(t)$ has nonnegative Hessian sectional curvature.

\emph{(iii)} If $g^T(0)$ has maximal volume growth, then $g^T(t)$ also has maximal volume\\ \hspace*{1.1cm}growth.
\end{theorem}

As an application, we can prove the topological rigidity theorem on noncompact Hessian manifolds:
\begin{theorem}
Let $(M^n,g_0)$ be a complete noncompact Hessian manifold with nonnegative Hessian sectional curvature, and its tangent bundle $(TM,g_0^T)$ has maximal volume growth. Then $M$ is diffeomorphic to $\mathbb{R}^n$.
\end{theorem}
\begin{proof}
By Theorem \ref{t4.1}, $M^n$ also supports a complete Hessian metric $g'$ with nonnegative and bounded Hessian sectional curvature, and its tangent bundle $(TM,g'^T)$ has maximal volume growth. Then, we can use Theorem \ref{thm2} instead of Corollary 5.4 in \cite{JY25}. Using the argument in \cite{JY25} \S 6, we can prove this theorem.
\end{proof}

\begin{lemma}\label{l4.3}
There exists $1>\alpha>0$ depending only on $n$ so that the following is true: Let $(N^n,g_0)$ be a Hessian manifold and $U\subset N$ is a precompact open set. Let $\rho>0$ be such that $B_0(x,\rho)\subset\subset N$, $|Q_{ijkl}|_{g_0}(x)\leq \rho^{-2}$ and ${\rm inj}_{g_0}(x)\geq \rho$ for all $x\in U$. Assume $U_\rho$ is nonempty. Then for any component $X$ of $U_\rho$, there is a solution $g(t)$ to the geometric flow \eqref{e1.2} on $X\times [0,\alpha \rho^2]$, where for any $\lambda>0$
\[U_\lambda:=\{x\in U|B_0(x,\lambda)\subset\subset U\},\]
with $g(t)$ satisfies the following:

\hspace*{0.1cm}\emph{(i)} $g(0)=g_0$ on $X$

\emph{(ii)}
\[\alpha g_0\leq g(t)\leq \alpha^{-1}g_0\]
\hspace*{1cm}on $X\times [0,\alpha \rho^2]$.
\end{lemma}
\begin{proof}
First, we find that this lemma holds for some $\rho=\widetilde{\rho}>0$ on $(N^n,g_0)$ if and only if it holds for $\rho=1$ on $(N^n,\widetilde{\rho}^{-2}g_0)$, so we can assume $\rho=1$. By the proof of Lemma 6.2 in \cite{r7}, there is a smooth function $\sigma(x)\geq 0$ on $U$ such that $\sigma(x)\equiv 0$ on $U_1$, $\sigma(x)\geq 1$ on $\partial U$, $|\nabla \sigma|+|\nabla^2 \sigma|\leq c_1$. In the proof of this lemma, we denote a positive constant depending only on $n$ by $c_i$. Define $W=\{x\in U|\sigma(x)<1\}$, then we have $X\subset U_1\subset W$. Let $\widetilde{W}$ be the component of $W$ containing $X$. Let $h_0=e^{2F}g_0$ be the Riemannian metric on $\widetilde{W}$ where $F(x)=\mathfrak{F}(\sigma(x))$ where $\mathfrak{F}$ is the function in \cite{LT20} Lemma 4.1. Note that $h_0^T=e^{2F\circ \pi}g_0^T$, by the proof of Lemma 5.1 in \cite{LT20}, $h_0^T$ is complete and has bounded curvature, the Chern curvature, the torsion $|T_0|_{h_0^T}$ and $|\bar{\partial}T_0|_{h_0^T}$ of $h_0^T$ are uniformly bounded by a constant $c_2$. By Theorem \ref{t3.4}, \eqref{e1.2} has a solution $h(t)$ on $\widetilde{W}\times [0,\alpha]$ for some $\alpha>0$ depending only $n$ such that
\[\alpha h_0\leq h(t)\leq \alpha^{-1}h_0\]
on $\widetilde{W}\times [0,\alpha]$. Let $g(t)=h(t)|_X$. Since $h_0=g_0$ on $X$ because $F=0$ there, it is easy that (i) and (ii) are true.
\end{proof}

\begin{lemma}\label{l4.4}
Suppose $(M,\nabla ,g)$ is a complete connected affine Riemannian manifold, it can be seen as a submanifold of its tangent bundle $(TM,J_\nabla,g^T)$. Then, we have:

\hspace*{0.1cm}\emph{(i)} $B^T(x,r)\cap M=B(x,r)$;

\emph{(ii)} $\pi(B^T(x,r))=B(x,r)$.\\
where $B(x,r)$ and $B^T(x,r)$ are the geodesic balls centered at $x$ with radius $r$ in $(M,\nabla ,g)$ and $(TM,J_\nabla,g^T)$ respectively.
\end{lemma}
\begin{proof}
Since $B(x,r)\in M$ and $g^T|_M=g$, we have $B(x,r)\subset B^T(x,r)\cap M$. By the proof of Lemma 6.1 and Proposition 6.3, we have
\[L_{g^T}(\gamma)\geq L_{g}(\pi(\gamma)),\]
where $\gamma$ is any curve on $(TM,J_\nabla,g^T)$. Then, $B^T(x,r)\cap M\subset B(x,r)$, we conclude that (i) is true.

By (i), we have $B(x,r)\subset \pi(B^T(x,r))$. From the above inequality, we obtain $\pi(B^T(x,r))\subset B(x,r)$, then (ii) follows.
\end{proof}

By Lemma \ref{l4.4}, the following Lemmas \ref{l4.5}-\ref{l4.8} can be derived directly from Lemmas 5.2-5.5 in \cite{LT20} using the relationship between the flow \eqref{e1.2} and the K\"{a}hler-Ricci flow, so we omit the proofs.

\begin{lemma}\label{l4.5}
For any $n,v_0>0$, there exist $\widetilde{S}(n,v_0)>0$, $C_0(n,v_0)>0$ depending only on $n,v_0$ such that the following holds: Suppose $(N^n,g(t))$ is a solution of \eqref{e1.2} on a Hessian manifold for $t\in [0,S]$ and $x_0\in N$ such that $B_t(x_0,r)\subset\subset N$ for each $t\in [0,S]$. Suppose
\[V_0(x_0,r)\geq v_0r^{2n}\;\;and\;{\rm BK}(g^T(t))\geq -r^{-2}\;on\;B_t(x_0,r),\;t\in[0,S]. \]
Then for all $t\in(0,S]\cap (0,\widetilde{S}r^2]$,
\[|{\rm Rm}(x,t)|\leq \frac{C_0}{t}\;\;and\;\;|\nabla Rm|(x,t)\leq \frac{C_0}{t^{3/2}}\;\;on\;B_t(x_0,\frac{r}{8}).\]
Moreover the injectivity radius satisfies
\[{\rm inj}_{g(t)}(x)\geq (C_0^{-1}t)^{\frac{1}{2}}\]
for $x\in B_t(x_0,\frac{r}{8})$ and $t\in [0,S]\cap [0,\widetilde{S}r^2]$.
\end{lemma}
\begin{lemma}\label{l4.6}
For any $A_1,n>0$, there exists $C_1(n,A_1)$ depending only on $n,A_1$ such that the following holds: For any Hessian manifold $(N^n,g_0)$(not necessarily complete), suppose $g(t),t\in [0,S]$ is a solution of \eqref{e1.2} on $B_0(x_0,r)$ where $B_0(x_0,r)\subset N$ such that on $B_0(x_0,r)$
\[|Rm_{g(0)}|\leq A_1 r^{-2}\;\;and\;\;|\nabla_{g(0)}{\rm Rm}(g_0)|\leq A_1 r^{-3}.\]
Assume in addition that on $B_0(x_0,r)$,
\[A_1^{-1}g_0\leq g(t)\leq A_1 g_0.\]
Then on $B_0(x_0,\frac{r}{8})$, $t\in[0,S]$,
\[|{\rm Rm}|(g(t))\leq C_1 r^{-2}.\]
\end{lemma}
\begin{lemma}\label{l4.7}
Suppose $(N^n,g(t))$ is a solution of \eqref{e1.2} on $[0,S]$ with $g(0)=g_0$, where $g_0$ is a Hessian metric. Let $x_0\in N$, and $r$ be such that $B_0(p,r)\subset\subset N$. Suppose there exists $a>0$ such that

\hspace*{0.1cm}\emph{(i)} $g_0$ has nonnegative Hessian sectional curvature on $B_0(x_0,r)$;

\emph{(ii)} $|{\rm Rm}|(g(t))\leq \frac{a}{t}$ on $B_0(x_0,r)\times (0,S]$.\\
There exists $\hat{S}>0$ depending only on $n,a$ such that for all $t\in[0,S]\cap [0,\hat{S}r^2]$, $x\in B_t(x_0,\frac{r}{8})$,
\[{\rm BK}(x,t)\geq -r^{-2}.\]
\end{lemma}
\begin{lemma}\label{l4.8}
There exists a constant $\beta=\beta(m)\geq 1$ depending only on $m$ such that the following is true. Suppose $(N^m,g(t))$ is a solution of \eqref{e1.2} for $t\in[0,S]$ and $x_0\in N$ with $B_0(x_0,r)\subset\subset N$ for some $r>0$, and $\beta(g(t))\leq (2m-1)a/t$ on $B_0(x_0,r)$ for each $t\in(0,S]$. Then
\[B_t(x_0,r-\beta \sqrt{at})\subset B_0(x_0,r).\]
\end{lemma}

We are ready to prove the theorem.
\begin{proof}[Proof of Theorem \ref{t4.1}]
The theorem can be proved using almost the same methods as Theorem 5.1 respectively in \cite{LT20} where
the complex case was considered, we only need to use Lemmas \ref{l4.3}, \ref{l4.5}-\ref{l4.8} instead of Lemmas 5.1-5.5 in \cite{LT20}, and use the relationship between the flow \eqref{e1.2} and the K\"{a}hler-Ricci flow, so we omit their proof.
\end{proof}



\end{document}